\documentclass[a4paper, final, 12pt]{article}

\usepackage{eurosym}
\usepackage{amsmath, amssymb, latexsym, amscd, amsthm,amsfonts,amstext}
\usepackage[mathscr]{eucal}
\usepackage{graphicx}
\usepackage{subfig}
\usepackage{float}
\usepackage{color}
\usepackage{hyperref}
\usepackage[utf8]{inputenc}
\usepackage[english]{babel}

\newtheorem{theorem}{Theorem}[section]
\newtheorem{lemma}{Lemma}[section]

\newtheorem{definition}{Definition}[section]

\newtheorem{corollary}{Corollary}[section]

\numberwithin{equation}{section}
\begin{document}

\title{
Elevator Optimization: \\
Application of Spatial Process and Gibbs Random Field Approaches for Dumbwaiter Modeling and Multi-Dumbwaiter Systems
}
\author{Zheng Cao$^1$, Benjamin Lu Davis$^2$, Wanchaloem Wunkaew$^3$, Xinyu Chang$^4$
\and University of Washington, Seattle, USA
\and Department of Mathematics
\and zc68@uw.edu $^1$
\and Department of ECE, Applied Math, Statistics
\and bldbld@uw.edu$^2$
\and Department of Applied Mathematics
\and leegarap@uw.edu$^3$
\and Department of Mathematics, Economics
\and xchang23@uw.edu$^4$
}
\date{}
\maketitle

\begin{abstract}
This research investigates analytical and quantitative methods for simulating elevator optimizations. To maximize overall elevator usage, we concentrate on creating a “multiple-user positive-sum system” that is inspired by agent-based game theory. We define and create basic “Dumbwaiter” models by attempting both the Spatial Process Approach and the Gibbs Random Field Approach. These two mathematical techniques approach the problem from different points of view: the spatial process can give an analytical solution in continuous space, and the Gibbs Random Field provides a discrete framework to flexibly model the problem on a computer. Starting from the simplest case, we target the assumptions to provide concrete solutions to the models and develop a “Multi-Dumbwaiter System.” This paper examines, evaluates, and proves the ultimate success of such implemented strategies to design the basic elevator's optimal policy; consequently, not only do we believe in the results' practicality for industry, but also their potential for application.

\end{abstract}

\vspace{0.5em}

\textbf{\text{Keywords:}}

Stochastic Optimal Control, Spatial Process, Gibbs Random Field, Agent-Based Game Theory, Operations Research, Markov Chain
\newpage

\tableofcontents

\newpage

\section{Introduction}

\hspace{1em}
“The Model Thinker”\cite{TMT} outlines the methodology for taking this language called “math,” which we humans invented and implemented in a way to describe reality in a manner that draws conclusive results to yield insightful wisdom to guide our future decision-making processes. We perform that analysis here for the elevator's journey of delivering passengers to help consider what the Bellman-Ford optimal policy should look like. 

\par
We have attempted to approach elevator optimization models in two ways: using spatial processing and Gibbs Random Field. This project is studied simultaneously as another research team completes “Application of Deep Q Learning with Simulation Results for Elevator Optimization” to apply machine learning and simulation to optimize elevator applications.\cite{ADQLSREO} Through our programming research team's work, we are applying machine learning techniques to optimize the elevator's allocations. Both projects produce independent results, but in future research, a better optimization result may be accomplished through a combination of both programming and math research results.

\section{Motivation}
\hspace{1em}
Inspired by Game Theory, this elevator optimization seeks to develop a “multiple-user positive-sum system” in addition to the “two-person zero-sum game” proved by John von Neumann. Here,
\begin{definition}
A “multiple-user positive-sum system” is a system in which a positive sum is created at the cost of a few users.
\end{definition}
In our case, we are attempting to mathematically improve the overall performance and efficiency of elevators by applying the Spatial Process and Gibbs Random Field Approaches. This procedure is quite analogous to agent-based modeling of complex adaptive systems.

Cao was really miffed at the excruciatingly long wait times for the esteemed University of Washington's Lander Hall dorms' inefficient elevator system, and thus vowed passionate commitment to devising a more efficient system using stochastic optimization and operations research. Increased elevator efficiency will aid in serving more customers and floor-traversing patrons, increasing the diversity, equity, inclusion, and disability accessibility of the building in question. The idea of combining programming and math modeling from parallel research is motivated by “Application of Convolutional Neural Networks with Quasi-Reversibility Method Results for Option Forecasting.”
Before shoving brute-force computation and deep machine learning neural networks onto the mathematical programming problem, we perform some preliminary, simple, underlying, first-level pencil-and-paper analysis to get a sense of the layout and lay the groundwork for how we should expect the system to behave.

\section{Assumptions}
\hspace{1em}
To begin with, three assumptions are set to develop the “Dumbwaiter” model for elevator optimization:

\begin{enumerate}
\item The model is restricted to one elevator in one building.

\item People do not ride the elevator directly to the same floor they're already on.

\item The elevator has unlimited human capacity (for all practical purposes).
\end{enumerate}

Notice that all three of the aforementioned assumptions are applicable for both approaches below. However, the Gibbs Random Field Approach requires some additional conditions or presumptive constraints.

Even though some of the above assumptions may sound unrealistic, e.g., there exists no such elevator with infinite capacity, nor any such building with a single elevator as the sole mode of transport, it is these purposefully implemented predetermined rules and slight limitations that help us to effectively utilize a rigorous pure-math approach to simplify the complex elevator usage model in the real world and its subsequent optimization procedure.

\section{Math Modeling - The “Dumbwaiter” Models}
\hspace{1em}
For those not familiar, a dumbwaiter is a small freight elevator specifically meant for delivering food to different floors directly from the kitchen, usually in fancy old buildings. In this paper, a “Dumbwaiter” is defined as the following:

\begin{definition}
A “Dumbwaiter System” is an elevator system that is the only transport medium and has unlimited human capacity, which exhibits the worst case handling of transporting human beings between floors of a certain building.
\end{definition}

The “dumbwaiter” models, hence the pun on the name, are meant to kind of characterize the worst-case scenario or examine delivery methodologies that completely lack any sort of intelligent planning. If the “dumbwaiter” lacks intelligence, it will clearly exhibit inefficiency in its trajectory, and thus performance metrics on the “dumbwaiter” models can be used as a direct benchmark for comparison to estimate and quantify any increased efficiency from these base-case scenarios using more intelligent elevator routing schemes. This is analogous to quantifying “information rent” from contract theory in direct comparison to the public-knowledge first-best-contract full-information benchmark case.

With these definitions and assumptions, we now are able to apply the Spatial Process Approach and Gibbs Random Field Approach to solve the “Dumbwaiter mystery.”

\section{Spatial Process Approach}
\hspace{1em}
The simple “dumbwaiter” is modeled as a spatial process.
The “Dumbwaiter,” hence the name, unintelligently directly visits each call in succession with no prior planning.
This is the informal name for this model, as it’s an amusing pun alluding to how old western buildings in the occident contained a miniature freight-elevator (with just enough space for a silver platter receptacle,) literally actually called a “dumbwaiter,” specifically for easy access to the basement chef kitchen to facilitate hoi polloi attendant/ servant/ butler/ maid delivery of food meals directly upstairs onward to the regally royal Prince/ King/ Lord/ Duke/ Baron/ Lady’s Hotel/ Palace master guest room.
This “dumbwaiter” model (simple enough solvable by hand) serves as a theoretical mundanely grounded baseline metric for direct comparison with future methods, as the rest of the project goes
(on average, dumbwaiter traverses 1/3 building each call under that simple proposed model).

\subsection{Notations}

\begin{enumerate}
\item $X_{i}$ - the floor the elevator receives a call from the $i$th group of people.
\item $Y_{i}$ - the destination floor of the elevator pressed by the $i$th group of people.
\item $T$ - total travel time of the elevator.
\item $L$ - total traversed distance of the elevator.
\item $Z$ - floor distances between the next movement of the elevator and previous location.
\item $R$ - absolute value of the floor distances of the two adjacent elevator movements.
\end{enumerate}

\subsection{Baseline Model}
\hspace{1em}
Each floor of “Hilbert's Hotel” is indexed in ascending order along the “Hotelling-Line” from $ [0, 1]$, and thus the total height of the building is normalized to unity \cite{HH}. The elevator receives a call from floor $X$, where random variable $X \sim f_X(x) = P(X=x) $, and the person or group of people entering on floor $X$ selects a destination floor by pushing the button to go to floor $Y$, where random variable $Y \sim f_Y(y) = P(Y=y) $.

\subsection{Further Assumptions} 
\begin{enumerate}
\item Each new entering group of people will only have one destination, and thus will only press the floor selection button once, implying there exists a bijective mapping of one single $ Y$ for each and every $X$. 
\item Random variables $X$ and $Y$ are I.I.D. with $f_X(x)=f_Y(x)$ and $X,Y \stackrel{I.I.D.}{\sim} f_X(x)$ 
\end{enumerate}
\hspace{1em}
(Note that these assumptions could indeed be relaxed for increased model complexity.)

\subsection{Solving this Model} 
\hspace{1em}
Suppose $n$ repeated i.i.d. draws are generated, where random vectors\\ 
\\
\centerline{ $ \vec{X}_{1:n}, \vec{Y}_{1:n} \stackrel{I.I.D.}{\sim} f_X(x) $,}\\
\par
and as such the “dumbwaiter” unintelligent elevator will visit each placed call in direct succession as:\\ 
\centerline{ $ X_1 \rightarrow Y_1 \rightarrow X_2 \rightarrow Y_2 \rightarrow X_3 \rightarrow Y_3 \rightarrow ... $}\\
\par
Without loss of generality since all data points were generated
I.I.D. by construction, we relabel the entire route’s trajectory information using $ \vec{X}_{1:2n} $, and thus we can focus on just the successive differences between each data point, as the status of whether it was a scheduled starting or destination call is irrelevant to our ultimate interest.
\par
Since the elevator travels at a constant velocity, the total travel time, $T$, must be proportional to the total traversed distance, $L$, written as $ T \propto L $. After normalizing the height of the building to unity, the aggregate total traversed distance of the elevator throughout its entire journey, must be directly proportional to the total incremental sum of each successive length along each leg of its entire journey, mathematically yielding the total sum of all absolute successive differences between each trajectory data point as mentioned previously, written as $L \propto \underset{k}{\sum} |X_{k+1} - X_k | $. Therefore, by finding the distribution of $L$, we can also understand the distribution of $T$. \\
\par 
Let's define a random variable $ Z_k \equiv X_{k+1}- X_k \rightarrow Z_k \sim f_Z(z) = P(Z=z) $, where \\
\\
\centerline{ $f_Z(z) = f_X(x)* f_X(-x)$,}\\
\par
where $*$ denotes convolution operation\cite{CO}, so $f_Z(z)$ is the auto-correlation of $f_X(x)$ with itself due to independence with increment of index $k$. 
\par
Let's define a random variable $R_k \equiv |Z_k|$. Clearly, the PDF of $R_k$ is the same as the PDF of $Z_k$ with all the negative half’s probability mass reflected back over to the positive side. 
\par
Given Laplace’s “principle of indifference” from Bayesian epistemology \cite{PI}, we assume a maximal
entropy distribution of $\vec{X}_{1:2n} \stackrel{I.I.D.}{\sim} U(0,1) $ standard uniform distribution. Subsequently it is clear to see that $R_k \sim P(R_k=r_k) = f_{R_k}(r_k) = 2 (1-r_k) I(0 \leq r_k \leq 1 )$ a right-triangle-shaped distribution, yielding $E[R_k] = \frac{1}{3}, Var(R_k) = \frac{1}{18} $.
\\
\begin{theorem}
Suppose that $R_k \sim P(R_k = r_k) = f_{R_k}(r_k) = 2 ( 1 -r_k) I (0 \leq r_k \leq 1)$.
The following are true:
\begin{enumerate}
    \item $E[R_k] = \frac{1}{3}$
    \item $Var(R_k) = \frac{1}{18}$
\end{enumerate}
\end{theorem}
\begin{proof}
\hspace{1em}
\par
Suppose $R_k$ is defined as above. Consider that the distribution of $R_k$ is in the triangle shape. Hence, the results are obtained.
\end{proof}
\par
Note how $R_{k} \equiv |X_{k+1}-X_k| $, and $R_{k-1} \equiv |X_k-X_{k-1} |$, meaning that successive absolute differences $R_{k}$ and $R_{k-1}$ Are not actually completely independent, implying some weak correlation between directly neighboring legs of the elevator’s journey. That being said, since the transition probability parameters remain constant throughout time, $R_k$ is a stationary random process. 
\par
Indeed, we can actually compute the correlation of each successive increment as follows:
\\
\begin{theorem}
Suppose that $R_k \sim P(R_k = r_k) = f_{R_k}(r_k) = 2 ( 1 -r_k) I (0 \leq r_k \leq 1)$.
Then, $E[R_k R_{k-1}] = \frac{7}{60}$.
\end{theorem}
\begin{proof}
Consider that
 \begin{align*} 
E[R_k R_{k-1} ] &= E[|X_{k+1}-X_k| | X_k-X_{k-1}|] \\
&= E[|X_k-X_{k+1}| | X_k-X_{k-1}|] \\
&= E \left[  E[ |X_k-X_{k+1}| | X_k-X_{k-1}|] \big{|} X_k = x_k \right]
\end{align*}
\\
By law of iterated expectations,
\begin{align*}
\hspace{2.5em}
&\left[ |X_k-X_{k+1} | \big{|} X_k=x_k \right] ,\left[ |X_k-X_{k-1} | \big{|} X_k=x_k \right]  \\
&= \left[|x_k - X_{k+1}|, | x_k - X_{k-1} | \right] \stackrel{I.I.D.}{\sim} f_{X_{k \pm 1} \big| X_k = x_k}(x_{k \pm 1}) \equiv P(X_{k \pm 1}= x_{k \pm 1} \big{|} X_k=x_k) \\
&= I(0 \leq x_{k \pm 1} \leq x_k) + I(0 \leq x_{k \pm 1} \leq 1 - x_k ) 
\end{align*}

\begin{align*} 
E[ |X_k-X_{k+1}| | X_k-X_{k-1}| \big{|} X_k = x_k ] &= E[ |x_k-X_{k+1}| | x_k-X_{k-1}| ] \\
&= E[ |x_k-X_{k+1}| ] E[ | x_k-X_{k-1}| ]  \\
&= E[ |x_k-X_{k+1}| ]^2 \quad \text{(by symmetry)}   \\
&= \left[ x_k \frac{x_k}{2} + (1-x_k) \frac{(1-x_k)}{2} \right]^2 \\
&= \left[ \frac{1}{2} \left[x_k^2 +(1-x_k)^2 \right] \right]^2 \\
&= x_k^4 -2 x_k^3 + 2x_k^2 -x_k + \frac{1}{4} 
\end{align*}

\begin{align*} 
\hspace{1.5em}
E \left[ E[ |X_k-X_{k+1}| | X_k-X_{k-1}| \big{|} X_k = x_k ] \right] &= \stackrel{1}{\underset{0}{\int}} x_k^4 -2 x_k^3 + 2x_k^2 -x_k + \frac{1}{4} d x_k  \\
&= \frac{7}{60}  \\
&= E[R_k R_{k-1} ]
\end{align*}
\end{proof}

Thus, the auto-covariance and auto-correlation can be obtained.
\\
\begin{corollary}
\begin{align*}
    Cov(R_k, R_{k-1}) &= \frac{1}{180} \\
    Corr(R_k, R_{k-1}) &= \frac{1}{10}
\end{align*}
\end{corollary}
\begin{proof}
\begin{align*}
Cov(R_k, R_{k-1}) \equiv E[R_k R_{k-1}] - E[R_k] E[R_{k-1}] &= \frac{7}{60} - \frac{1}{3} \times \frac{1}{3} \\
&= \frac{7}{60} - \frac{1}{9}  \\
&= \frac{1}{180}\\
Corr(R_k, R_{k-1}) \equiv \frac{Cov(R_k, R_{k-1})}{\sqrt{Var(R_k) \times Var(R_{k-1}) }} &= \frac{\left[ \frac{1}{180} \right] }{\sqrt{ \left[ \frac{1}{18} \right] \times \left[ \frac{1}{18} \right]}}  \\
&= \frac{\left[\frac{1}{180} \right]}{ \left[\frac{1}{18} \right]}\\
&=  \frac{1}{10} \\
&= 0.1 
\end{align*}
\end{proof}

Notice that $R_{k+1} \equiv |X_{k+2} - X_{k+1}| $ and $ R_{k-1} \equiv |X_{k} - X_{k-1}| $, implying $R_{k-1}$ and $R_{k+1}$ are completely independent of each other. Therefore, we have an $m$-dependent sequence for $m=2$ in this case, indicating a process exhibiting the Strong Markov Property, where intuitively a time-neighboring index separation of $m=2$ “refreshes” this process to start anew in sense. So, with increasingly large index $k$, we can aptly invoke the Law of Large Numbers, Central Limit Theorem, and Berry-Esseen-type Bounds on the empirically observed sample mean of the stochastic process of $R_k$ to deduce asymptotic measurements and statistical properties to characterize the elevator’s route, hopefully yielding insightful intuition for improvement upon this unintelligent “dumbwaiter” model. 

\subsection{Apply to the Real-World}
\hspace{1em}
This implies that on average over long time periods, this “dumbwaiter“ baseline model should asymptotically traverse $ \frac{1}{3}$ of the height of the building per call, with a variance of $ \frac{1}{18} $ fluctuation around this expected value. Since we have $T \propto L $, and $L = T \times $ velocity, if the elevator is traveling at a speed of $45 $ meters per minute\cite{PME}, we can calculate the average distance the elevator will need to travel per call. Suppose we are looking at a 10-story building, with each floor being length $4.2$ meters\cite{SD}, then the whole building will have a total height of $42$ meters. On average, each elevator call will make the elevator travel about $14$ meters. Then, we can calculate that $ T= \frac{L}{velocity} = \frac{14}{45} \approx 0.3$ minute $ = 18 $ seconds. Thus, for each call, the elevator will on average have to spend approximately 18 seconds to complete its service.
\par
Clearly, this model shows improvement over the case involving the elevator returning back to the ground floor between customers. Indeed, there is certainly a way for further improvement on this model; as long as this “dumbwaiter” beforehand more intelligently plans ahead its journey and visitation schedule, the elevator will waste less time, and thus intelligent methods implementing information acquisition will allow the elevator to travel less distance to serve more people efficiently. \\

Furthermore, it is impossible to further improve:

\begin{lemma}
Lemma: There exists a slightly intelligent waitress who can only beat and improve upon the dumbwaiter.

\end{lemma}
Enormously simple mathematical justification:
\par
Suppose the dumbwaiter visits $ X_a \rightarrow X_{b} \rightarrow X_{c} $ in direct succession without altering the order. The slightly intelligent waitress knows from her classic secretary problem training to wait and receive calls $X_a, X_{b}, X_{c} $ before performing the elevator-routing-problem.
\par
The dumbwaiter's length of path or time of journey will $ \propto | X_{c} - X_{b} | + |X_{b} -X_a | $.
The slightly intelligent waitress's length of path or time of journey will be proportional to:
\begin{align*}
\max{[X_a, X_b, X_c]} - \min{[X_a, X_b, X_c]} &= range[X_a, X_b, X_c] \\
&= \max{[ |X_a-X_{b}| , |X_{b}-X_{c} |, |X_c-X_a|]}
\end{align*}
which by triangle inequality, cannot be worse than the dumbwaiter's plan.
\par
The amount of improvement in the slightly intelligent waitress's procedure versus the dumbwaiter could potentially be estimated using the order statistics of the samples if the model implemented assumes independence, which would require the original I.I.D. sample's CDF.

\section{Gibbs Random Field Approach}
\hspace{1em}
Gibbs Random Field Approach is the groundwork for the deep reinforcement machine learning stuff, but the notion of implementing decision nodes for transitioning between global ensemble states serves as the theoretical framework and construct for how to optimize and tell the elevator what to do and learn from repeated experience, etc\cite{GR}. Hopefully, this will outperform the previous “dumbwaiter” and get the food to the bourgeoisie upstairs on time.

In short, we use discrete math to define a Markov Chain random walk on a Markov Graphical Model:
\par
All possible permutations, configurations, or arrangements of the aggregate system canonical ensemble of elevators and humans are represented in the state space or statistical mechanics phase space. We assume no degeneracy of states' energy levels, implying Fermi-Dirac counting. We describe the total current state of the system and fully represent all information about configuration entropy with a binary vector of indicator functions of occupancy. In essence, elevator motion and delivery will cause the transport of people from one location to another, causing the aggregate system to jump from one overall state to another. This calls for modeling the behavior of a Markov Chain random walk on a graph, meaning each time the elevator makes a trip to deliver people, we have one time unit of incremental dynamic progression on the Markov Chain, and the type of delivery will dictate how the system hops from one state to another. It will become a reinforcement learning problem to figure out the optimal way to select which states to jump to and from to steer the trajectory. To create a well-formulated optimization problem that's not ill-posed, we implement a directed acyclic weighted graph (an irreducible and periodic Markov chain) including a marginal objective penalization factor based on choice of route. Clearly the amount of penalization must be loosely commensurate to various physical factors and practical considerations, such as aggregate womyxn-man-hour wait time, opportunity cost, node affinity, number of door-opening stops, arc length of an elevator's parameterized curve trajectory, or geometric geodesic spatial configuration distance between the stops or floors within the context of the global system's aggregate states, analogous to graphical Laplacian implementations using the spatial arrangement of finite element meshes. 
\par
In this paper, we introduce a discrete-time Markov Chain to model elevator (or dumbwaiter) control. With this framework, the elevator is allowed to move randomly with respect to a pre-specified probability measure or state-transition probability of the stochastic process. This generalized non-deterministic process allows us the flexibility to design elevator control policies. Furthermore, because of the simplicity of the Markov Chain model, we can evaluate the model analytically. This analytic and quantitative evaluation system allows us to assess and compare the efficiency of different models; hence, we can see the model as an optimization problem.
In this paper, we provide two baseline models with some model evaluation in order to illustrate the efficiency of the framework.

We set the sample building to have $N$ floors, for $N \in \mathbb{N}$.

\subsection{Notations}
\begin{enumerate}
    \item $\{X_t, t \in \mathbb{N}\}$ : A discrete-time stochastic process, which is Markov Chain in the baseline model.
    \item $\{Y_t, t \in \mathbb{N}\}$ : A discrete-time stochastic process, which is Markov Chain in the Floor Waiting Time model.
    \item $\tau_i$ : The first hitting of the Markov chain to some state $i$.
\end{enumerate}
\subsection{Baseline model}
\subsubsection{Additional Assumptions}
\begin{enumerate}
    \item The building has $N$ floors, for $N \in \mathbb{N}$.
\end{enumerate}
\hspace{1em}
Consider a Markov Chain,$X_t$, $t \in \mathbb{Z}^+$, representing a floor that the elevator visits. In this case, we have N states for this chain, since the building has $N$ floors.
\par
Moreover, we define the transition probability from state $a$ to $b$ (eg. elevator moving from floor $3$ to $4$) by $P_{a,b} = P(X_{t+1} = b|X_t = a)$. The following properties of chains are observed:
\begin{enumerate}
    \item The Markov chain has reflective boundary condition (the boundaries are $1$ and $N$). In other words, the elevator cannot go down below the first floor or fly over the $N$th floor.
    \item Suppose $x \in {2,3,4,..., N-1}$. We have $P(X_{t+1} = x+1|X_t = x) > 0$, $P(X_{t+1} = x-1|X_t = x) > 0$, and $P(X_{t+1} = x+1|X_t = x) + P(X_{t+1} = x-1|X_t = x) + P(X_{t+1} = x|X_t = x)  = 1.$ That is, in each time step, the elevator can only move to floors connected to the floor it is on.
\end{enumerate}

To make the Markov Chain irreducible, we define the transition probabilities to the connected floor as strictly non-zero.

Given any transition probability constraint to the conditions above and given an initial state (the Markov Chain current floor), we can calculate the expected first hitting time of the Markov, which is the expected time that the elevator will move to some specific floor for the first time. We illustrate this point in the next model.

\subsection{A Model Pertained To Floor Waiting Time}
\subsubsection{Additional Assumptions}
\begin{enumerate}
    \item The elevator can hold an infinite number of people.
    \item Passengers leave the elevator immediately as the elevator goes past the target floor. 
    \item The time taken by each passenger to get in or get out of the elevator is negligible.
    \item Passengers are only aware of the time they wait for the elevator to arrive.
\end{enumerate}
\hspace{1em}
The chain defined in the previous subsection does not account for the time people on each floor waiting the elevator to arrive.
\par
In order to cope with such problem, we introduce a new set of states to the model. 

Consider that the elevator cannot operate on multiple floors at the same time. However, it is possible for people on each floor to wait for such an elevator at the same time. Suppose that the probability that a person on the $i$th floor would push a calling button at each time step is $p_i$. Since we have to consider both the floor the elevator is on and floors where people are waiting, we separate the state of the markov chain into two different sub-states:

\begin{enumerate}
    \item The floor the elevator is on, c.\\
    This sub-state takes $1,2,3,...,N$ representing the current floor the elevator is on.
    \item Floors people are waiting for the elevator.\\
    We define $w \in {0,1}^N$, with $w_i$ representing whether or not people are waiting on the floor $i$th. To illustrate, consider $N = 3$ and $w = (1,0,1)$, this means that people are waiting on the first and the third floors.
\end{enumerate}
\hspace{1em}
Combing two sub-state together, we can define $(c,w)$ to be a state for the Markov chain. Note that for an $N$ floor building, the model takes $N2^N$ states. Despite its exponential complexity, this model works in real-world practice as many buildings do not have too many floor.
\par
Remark that the properties and assumptions pertaining to the elevator movement suggested in the previous section remain in this model. However, there are some other conditions we propose in this  model:
\begin{enumerate}
    \item Let $P(Y_t = (c_t, w_t)) = 0$ such that $c_t = c$ and $(w_t)_c = 1$. In other word, people cannot push a call button if the elevator is on such a floor.
    \item As we have just assumed, a person on each floor would put the button independently. We have the probability that people on $i_1, i_2,...,i_m$  floor would put the button, given that none of the buttons on such floor is push, is $p_{i_1} \cdot p_{i_2} \cdot ... \cdot p_{i_m}$.
\end{enumerate}
\hspace{1em}
The usefulness of this model is that we can evaluate the expected waiting time of people on each floor. \\

Before jumping into the evaluation, these are additional suppositions to be satisfied in order to apply the model:

Consider the states $s_c = (c, 0,0,0,...,0)$, $c \in \{1,2,3,...,c\}$. It is a state that the elevator is on floor $c$ and there are no people waiting on any floor. Supposedly, the elevator works fast enough so that there are some points at which there are no people waiting for it; otherwise, people are less likely to wait for an elevator. Assume that the elevator is at some state, let $\tau_i$ be the first hitting time to the state that the elevator is on $i$th floor and there is no people waits on any floor.
\par
Then, the optimization problem is to minimize $E[\sum_i \tau_i]$. Formally,
$\min_{P} \{E[\sum_i \tau_i]\}$. In other words, we have to find the transition probability such that the first hit from some floor to every floor is minimized. The transition probability matrix can be created using analytical modelling or heuristic algorithms such as the Genetic Algorithm. This framework provides plausible methods to evaluate such a matrix.

\section{Multi-Dumbwaiter System}

\begin{definition}
A “Multi-Dumbwaiter System” is a system of multiple Dumbwaiter systems, for which contains multiple elevators in one building system. Elevators still function as the only transport media, with limited human capacity $n$, and exhibit the worst case handling of transporting human beings between floors of a certain building. 

\end{definition}

Recall the three assumptions set up earlier:
\begin{enumerate}
\item The model is restrict to one elevator in one building.
\item People only use elevators for moving through different floors in the building.
\item The elevator has unlimited human capacity.
\end{enumerate}

Now, with both the Spatial Process Approach and the Gibbs Random Field Approach completed, we are able to loosen the restrictions of the first and third assumptions:

\begin{enumerate}
\item The model includes $m$ number of elevators in one building.
\item People only use elevators for moving through different floors in the building.
\item Each elevator has a limited human capacity of $n$.
\end{enumerate}

Under some circumstances:

\begin{theorem}

We can equally distribute over-populated passengers into the $m$ different number of elevators so that each elevator inside this “Multi-Dumbwaiter System” system forms a “Dumbwaiter System”, depending on the total number of passengers $A$, the elevator capacity $n$, and the number of elevators $m$.

\end{theorem}

As such, we are able to achieve a preliminary simplification to a realistic, and thus more complicated “Multi-Dumbwaiter” elevator optimization.

\section{Summary}
\hspace{1em}
In this paper, we study the elevator optimization problem via dumbwaiter modeling with two approaches: Spatial Process and Gibbs Random Field (Markov Chain).
\par
In the Spatial Process approach, we consider a dumbwaiter movement in continuous
space. By analyzing the distribution of the absolute differences between the elevator's adjacent movements, we expect that for each elevator call from a group of people or a single person, the elevator will need to travel $\frac{1}{3}$ the height of the building to complete its job. This means, in the real world, we expect this elevator to spend roughly 18 seconds completing each of its services on average.
\par
On the other hand, in the Gibbs Random Field approach, we introduce a framework which allows us to flexibly design a stochastic elevator control policy and analytically evaluate such a policy. We also give a model that pertains to the time passengers have to wait for the elevator to arrive at their floor.

\par
In addition, we defined “Multi-Dumbwaiter System” and some corresponding theory strategies to aid in the processing of buildings with multiple floors.
A follow-up project will be combining Reinforcement Q Learning and simulation results with the two dumbwaiter model approaches to optimize elevator usage for more complicated situations.

\section{Acknowledgement}

\hspace{1em}
We thank Dr. Zhen-Qing Chen\footnote{Professor, Department of Mathematics, University of Washington, Seattle}, Dr. Soumik Pal\footnote{Professor, Department of Mathematics, University of Washington, Seattle}, and Dr. Kirill V. Golubnichiy \footnote{Postdoctoral Researcher, Department of Mathematics and Statistics, University of Calgary}for comments that greatly improved the manuscript.

We also acknowledge the help, from the authors of the additional paper “Application of Deep Q learning with Simulation Results for Elevator Optimization”, have serviced in creating further applications based on the math approaches mentioned in this paper: Zheng Cao, Raymond Guo, Caesar M. Tuguinay, Mark Pock, Jiayi Gao, and Ziyu Wang.


\end{document}